\documentclass[12pt,a4paper]{article}

\usepackage{epsf,epsfig,amsfonts,amsgen,amsmath,amstext,amsbsy,amsopn,amsthm}
\usepackage{amsmath,times,mathptmx}
\usepackage{amsfonts,amsthm,amssymb}
\usepackage{graphicx}
\usepackage{latexsym,bm}
\usepackage{indentfirst}
\usepackage{color}
\usepackage[colorlinks=true,anchorcolor=blue,filecolor=blue,linkcolor=blue,urlcolor=blue,citecolor=blue]{hyperref}
\usepackage{float}
\usepackage{tikz}
\usepackage{verbatim}
\usepackage{mathrsfs}

\evensidemargin=\oddsidemargin
\newtheorem{thm}{Theorem}[section]

\newtheorem{ques}{Question}[section]

\newtheorem{cor}{Corollary}[section]

\newtheorem{exam}{Example}
\newtheorem{conjecture}{Conjecture}[section]
\newtheorem{claim}{Claim}[section]

\addtocounter{section}{0}

\begin{document}

\title{Advances on two spectral conjectures regarding booksize of graphs\footnote{Supported by the National 
Natural Science Foundation of China (No.\! 12571369) and the Natural Science Foundation of Shanghai (No.\! 25ZR1402390).}}

\author{{\bf Mingqing Zhai$^a$},\quad {\bf Rui Li$^b$},\quad
{\bf Zhenzhen Lou$^c$}\thanks{Corresponding author. E-mail addresses: mqzhai@njust.edu.cn (M. Zhai), lirui@hhu.edu.cn (R. Li),
 louzz@usst.edu.cn (Z. Lou).}
\\
{\footnotesize $^a$ School of Mathematics and Statistics, Nanjing University of Science and Technology}\\
{\footnotesize Nanjing 210094, Jiangsu, China}\\
{\footnotesize $^b$ School of Mathematics, Hohai University,
         Nanjing 211100, Jiangsu, China}\\
{\footnotesize $^c$ College of Science, University of Shanghai for Science and Technology}\\
{\footnotesize Shanghai 200093, China}
}
\date{}

\maketitle

\begin{abstract}
The \emph{booksize} \( \mathrm{bk}(G) \) of a graph \( G \),
introduced by Erd\H{o}s, refers to the maximum integer \( r \) for which $G$ contains the book \( B_r \) as a subgraph. 
This paper investigates two open problems in spectral graph theory related to the booksize of graphs.

First, we prove that for any positive integer $r$ and any \( B_{r+1} \)-free graph \( G \) with \( m \geq (9r)^2 \) edges, 
the spectral radius satisfies \( \rho(G) \leq \sqrt{m} \). 
Equality holds if and only if \( G \) is a complete bipartite graph. 
This result improves the lower bound on the booksize of Nosal graphs (i.e., graphs with \( \rho(G) > \sqrt{m} \)) 
from the previously established \( \mathrm{bk}(G) > \frac{1}{144}\sqrt{m} \) to \( \mathrm{bk}(G) > \frac{1}{9}\sqrt{m} \), 
presenting a significant advancement in the booksize conjecture proposed Li, Liu, and Zhang.

Second, we show that for any positive integer $r$ and any non-bipartite \( B_{r+1} \)-free graph \( G \) with \( m \geq (240r)^2 \) edges, 
the spectral radius $\rho$ satisfies \(\rho^2<m-1+\frac{2}{\rho-1}\),
unless $G$ is isomorphic to $S^+_{m,s}$ for some $s\in\{1,\ldots,r\}$. 
This resolves Liu and Miao's conjecture and further reveals an interesting phenomenon: even with a weaker
spectral condition, $\rho^2\geq m-1+\frac2{\rho-1}$, we can still derive the supersaturation of the booksize for non-bipartite graphs.

\end{abstract}

\begin{flushleft}

\end{flushleft}

\section{Introduction}\label{se-1}
A \emph{book of size $r$}, denoted by \( B_r \), is the graph formed by \( r \) triangles sharing a common edge. 
The \emph{booksize} \( \mathrm{bk}(G) \) of a graph \( G \) is defined as the maximum integer \( r \) for which \( B_r \) is a subgraph of \( G \). 
The investigation of books in graphs has a rich history in extremal graph theory, originating from the seminal work of Erd\H{o}s~\cite{Erdos}, 
who established that every graph on \( n \) vertices with \( \lfloor n^2/4 \rfloor + 1 \) edges satisfies \( \mathrm{bk}(G) = \Theta(n) \). 
Erd\H{o}s further conjectured that \( \mathrm{bk}(G) \geq n/6 \), 
a result subsequently proved by Edwards (unpublished, see~\cite[Lemma 4]{E-F-C}) and independently by Khad\v{z}iivanov and Nikiforov~\cite{K-N}. 
Alternative proofs were later provided by Bollob\'{a}s and Nikiforov~\cite{B-N} and Li, Feng, and Peng~\cite{L-F-P}. 
For additional results on books, we refer to~\cite{C-F-S,C-F-W,Zhai-Lin} and the references therein. 
In this paper, all graphs under consideration are simple, connected, and possess \( m \) edges with no isolated vertices.

In spectral graph theory, a fundamental result was established by Nosal~\cite{Nosal-1}: 
every triangle-free graph with \( m \) edges satisfies \( \rho(G) \leq \sqrt{m} \). 
Nikiforov~\cite{Nikiforov2002} strengthened this by demonstrating that if \( \rho(G) \geq \sqrt{m} \), 
then \( G \) contains a triangle unless it is a complete bipartite graph. 
This spectral analogue of Mantel's theorem was extended by Lin, Ning, and Wu~\cite{Lin-Ning-Wu} to graphs with \( \rho(G) \geq \sqrt{m-1} \).

A graph \( G \) with \( m \) edges is called \emph{Nosal} if \( \rho(G) > \sqrt{m} \) (see~\cite{L-L-Z-1}). 
Nosal graphs exhibit intriguing structural characteristics beyond those mentioned above. 
Nikiforov~\cite{Nikiforov2009} showed that every Nosal graph of size at least 10 contains a 4-cycle unless it is a star. 
Zhai, Lin, and Shu~\cite{Zhai-Lin-Shu} proved that every Nosal graph contains a copy of \( K_{2,r} \) with \( r \geq \sqrt{m/4} + 1 \) unless it is a star. 
This bound was recently improved to an asymptotically optimal form, \( r \geq \sqrt{m/2} - O(1) \), by Li, Liu, and Zhang~\cite{L-L-Z-2}.

Given that a book \( B_r \) can be constructed from \( K_{2,r} \) by adding an edge, 
Zhai, Lin, and Shu~\cite{Zhai-Lin-Shu} conjectured that Nosal graphs contain large books. 
Using the method of counting walks and an algorithm for deleting edges with small Perron coordinates, 
Nikiforov~\cite{Nikiforov-book} confirmed Zhai, Lin, and Shu's conjecture by establishing the following result:

\begin{thm}[\cite{Nikiforov-book}]\label{Niki}
Every Nosal graph \( G \) satisfies \[ \mathrm{bk}(G) > \frac{1}{12}\sqrt[4]{m}. \]
\end{thm}

Nikiforov remarked that this bound appears far from optimal, as both the multiplicative constant and the exponent \( 1/4 \) might be improved. 
Subsequently, Li and Peng~\cite{Li-Peng} conjectured that the correct exponent is \( 1/2 \), which, if true, would be best possible.

\begin{conjecture}[\cite{Li-Peng}]
Every Nosal graph \( G \) satisfies \( \mathrm{bk}(G) = \Omega(\sqrt{m}). \)
\end{conjecture}

Using a probabilistic method, Li, Liu and Zhang~\cite{L-L-Z-1} confirmed the above conjecture
by proving the following result:

\begin{thm}[\cite{L-L-Z-1}]\label{Main-thm-0}
Every Nosal graph \( G \) satisfies
\[
\mathrm{bk}(G) > \frac{1}{144} \sqrt{m}.
\]
\end{thm}

\begin{exam}[\cite{L-L-Z-1}]\label{exam}
Let \( C_3^\square \) be the triangular prism consisting of two disjoint triangles and a perfect matching joining them.
Define \( G = C_3^\square [k] \) as the \( k \)-blow-up of \( C_3^\square \), i.e.,
by replacing each vertex of \( C_3^\square \) with an independent set of size \( k \) and each edge of \( C_3^\square \) with
a complete bipartite graph \( K_{k,k} \).
Then, \(|G| = 6k \), \(m= 9k^2 \), and \( \text{bk}(G) = k = \frac{1}{3}\sqrt{m} \).
Since $G$ is $3k$-regular, it is easy to see that \( \rho(G)=3k =\sqrt{m} \).
\end{exam}

Motivated by Theorem~\ref{Main-thm-0} and the example above, Li, Liu, and Zhang proposed the following question:

\begin{ques}[\cite{L-L-Z-1}]
What is the optimal constant \( C \) such that every Nosal graph contains a book of size \( C\sqrt{m} \)? Is it \( C = \frac{1}{3} \)?
\end{ques}

Using methods entirely distinct from those of Nikiforov~\cite{Nikiforov-book} and Li, Liu, and Zhang~\cite{L-L-Z-1}, 
we establishing the following spectral extremal result:

\begin{thm}\label{Main-thm-1}
Let \( G \) be a \( B_{r+1} \)-free graph with \( m \) edges, where \( m \geq (9r)^2 \). Then 
\(
\rho(G) \leq \sqrt{m},
\)
with equality if and only if \( G \) is a complete bipartite graph.
\end{thm}

An immediate corollary yields an improved bound on the booksize of Nosal graphs:

\begin{cor}\label{cor-book}
Every Nosal graph \( G \) satisfies
\[
\mathrm{bk}(G) > \frac{1}{9} \sqrt{m}.
\]
\end{cor}

We further investigate non-bipartite \( B_{r+1} \)-free graphs. 
Let \( S^{+}_{m,s} \) denote the graph constructed by embedding an edge into the \( \frac1s(m-1) \)-vertex part of \( K_{s,(m-1)/s} \). 
In particular, when \( s = 1 \), \( S^{+}_{m,1} \) is obtained by adding an edge within the independent set of \( K_{1,m-1} \). 
Liu and Miao~\cite{Liu-Lu} proposed the following conjecture:

\begin{conjecture}[\cite{Liu-Lu}]
Let \( G \) be a non-bipartite \( B_{r+1} \)-free graph with \( m \) edges. Then
\[
\rho(G) \leq \rho(S^{+}_{m,1}),
\]
with equality if and only if \( G \cong S^{+}_{m,1} \).
\end{conjecture}

Let \( \rho_0 = \rho(S^{+}_{m,s}) \). As will be discussed, it can be shown that
\(\rho_0^2 = m - 1 + \frac{2s}{\rho_0 - 1}.\)
We resolve Liu and Lu's conjecture in the following form:

\begin{thm}\label{Main-thm-2}
Let \( G \) be a non-bipartite \( B_{r+1} \)-free graph with \( m \geq (240r)^2 \) edges and spectral radius \( \rho \). Then
\[
\rho^2 < m - 1 + \frac{2}{\rho - 1},
\]
unless \( G \cong S^{+}_{m,s} \) for some \( s \in \{1, 2, \ldots, r\} \).
\end{thm}

The remainder of this paper is organized as follows.
Section~\ref{sec-2} offers a proof of Theorem~\ref{Main-thm-1}, 
while Section~\ref{sec-3} lays out the proof of Theorem~\ref{Main-thm-2}. 
Finally, we conclude with remarks and an open question in Section~\ref{sec-4}.

\section{Proof of Theorem \ref{Main-thm-1}}\label{sec-2}

Let $G$ be a $B_{r+1}$-free graph with $m$ edges. 
For simplicity, write $\rho = \rho(G)$. 
Let $\mathbf{x}$ be the Perron vector of $G$, with coordinates $x_v$ for $v \in V(G)$. 
Let $u^*$ be a vertex with maximum Perron weight, i.e., $x_{u^*} = \max\{x_v : v \in V(G)\}$. Define:
\begin{align*}
U=N(u^*)~ \mbox{and} ~
W=V(G)\setminus(\{u^*\}\cup U).
\end{align*}
Since $G$ is $B_{r+1}$-free, every vertex in $U$ has at most $r$ neighbors in $U$. The eigenvalue equation gives:
\begin{equation}\label{eq-2-1}
\rho^2 x_{u^*} = d(u^*) x_{u^*}\!+\!\!\!\sum_{uv \in E(U)}\!\!\!(x_u\!+\!x_v)+\!\sum_{w \in W}\!\!d_U(w) x_w.
\end{equation}

Define $f(w)=d_U(w)(x_{u^*}\!-\!x_w)\!+\!\frac{1}{2}d_W(w)x_{u^*}$. 
Notice that $e(W)=\frac{1}{2}\sum_{w \in W} d_W(w)$.
It immediately follows that
\begin{equation}\label{eq-2-2}
\rho^2 x_{u^*} = \big(m-e(U)\big)x_{u^*}\!+\!\!\!\sum_{uv \in E(U)}\!\!\!(x_u\!+\!x_v)-\!\sum_{w \in W}\!\!f(w).
\end{equation}

We now classify the edges in $E(U)$. An edge $uv \in E(U)$ is called \emph{bad} if $x_u + x_v \geq x_{u^*}$. 
Let $E^*(U)$ be the set of bad edges within $U$. For each $u \in U$, define:
\[
E^*_u = \{ uv \in E^*(U) : x_u \geq x_v \}.
\]
(If $x_u = x_v$, assign the edge arbitrarily to either $E^*_u$ or $E^*_v$). 
Then $E^*(U) = \bigcup_{u \in U} E^*_u$ forms a partition, where $E^*_u$ may be empty for some vertices $u\in U$. 
Since $G$ is $B_{r+1}$-free, we have $|E^*_u| \leq d_U(u) \leq r$ for all $u \in U$.

For $u\in U$ and $w\in W$, we define $N_W(u)=N(u)\cap W$ and $d_W(u)=|N_W(u)|.$
Let $W^* = \{ w \in W : x_w \geq (1 - \frac{4.5r}{\rho}) x_{u^*} \}$. 
A vertex $u \in U$ is called \emph{bad} if $E^*_u \neq \emptyset$. 
Let $c\geq1$ be a constant. We further classify bad vertices as:
\begin{align*}
U_1 &= \{ u \in U : E^*_u \neq \emptyset,\ d_{W \setminus W^*}(u)>c\rho/4.5 \}, \\
U_2 &= \{ u \in U : E^*_u \neq \emptyset,\ d_{W \setminus W^*}(u)\leq c\rho/4.5 \}.
\end{align*}
Note that $x_u \geq \frac12x_{u^*}$ for any $u \in U_1 \cup U_2$.
Partition $E(U)$ into:
\begin{align*}
E_1 \!=\! \bigcup_{u \in U_1} E^*_u, ~
E_2 \!= \!\bigcup_{u \in U_2} E^*_u, ~ \mbox{and} ~
E_3 \!= E(U) \setminus (E_1 \cup E_2).
\end{align*}
Observe that $|E_1| \leq r|U_1|$. For any $uv \in E_1$, we have $x_u + x_v \leq 2x_{u^*}$, so:
\begin{equation}\label{2-eq-3}
\sum_{uv \in E_1}\!\! (x_u\!+\!x_v)\leq2|E_1| x_{u^*}\leq|E_1| x_{u^*}\!+\!r|U_1| x_{u^*}.
\end{equation}

Let $V^* = \{ v \in U : uv \in E^*_u \text{ for some } u \in U_2 \}$,
and let $\max_{u\in N_{U_2}(v)}x_u=\beta_vx_{u^*}$. 
Clearly, $\frac12\leq\beta_v\leq1$ and $x_v\leq\beta_vx_{u^*}$ for each $v\in V^*$.
Moreover, we have
\[
\sum_{uv \in E_2}\!\!(x_u\!+\!x_v)\!\leq\!\sum_{uv \in E_2}\!\!(\beta_vx_{u^*}\!+\!x_v)
\!=\!|E_2| x_{u^*}\!+\!\!\sum_{uv \in E_2}\!\!\big((\beta_v\!-\!1)x_{u^*}\!+\!x_v\big).
\]
Since $x_{u^*}\geq \frac1{\beta_v}x_v$ and $d_{U}(v)\leq r$ for any $v \in V^*$, it follows that
\begin{equation}\label{2-eq-4}
\sum_{uv \in E_2}\!\!(x_u\!+ x_v)\leq|E_2| x_{u^*}\!+\!\!\sum_{v \in V^*}\!\!\big(2\!-\!\frac1{\beta_v}\big)rx_v.
\end{equation}
For $uv \in E_3$, we have $x_u + x_v < x_{u^*}$, so:
\begin{equation}\label{2-eq-5}
\sum_{uv \in E_3} (x_u + x_v) < |E_3| x_{u^*}.
\end{equation}

We now bound the contribution from $W \setminus W^*$.

\begin{claim}\label{clm-2.1} We have
\[
\sum_{w \in W \setminus W^*}\!\!f(w)\geq c r |U_1|x_{u^*}
\]
Furthermore, if $U_1$ is non-empty, then the above inequality is strict.
\end{claim}

\begin{proof}
Since each $u \in U_1$ has more than $\frac{c \rho}{4.5}$ neighbors in $W \setminus W^*$, it follows that
\[
\sum_{w \in W \setminus W^*}\!\!\!\!\!d_U(w)\geq\!\!\!\!\!\sum_{w \in W \setminus W^*}\!\!\!\!\!d_{U_1}(w) \geq \frac{c\rho}{4.5}|U_1|.
\]
where the last inequality is strict provided that $U_1\neq\emptyset$.
For $w \in W \setminus W^*$, we know that $x_w \leq (1\!-\!\frac{4.5r}{\rho}) x_{u^*}$. Thus:
\[
\sum_{w \in W \setminus W^*}\!\!d_U(w)(x_{u^*}-x_w) 
\geq \frac{4.5r}{\rho}\!\!\!\sum_{w \in W \setminus W^*}d_U(w)\geq c r |U_1|x_{u^*}.
\]

Recalling the definition of \(f(w)\) and rearranging terms, we obtain:
\[
\sum_{w \in W \setminus W^*}\!\!f(w)\geq \sum_{w \in W \setminus W^*}\!\!d_U(w)(x_{u^*}-x_w) \geq c r |U_1|x_{u^*}
\]
This establishes the claim. When \(U_1 \neq \emptyset\), the strictness of the inequality follows from the strict inequality in our neighbor count estimate.
\end{proof}

Combining~\eqref{eq-2-2}--\eqref{2-eq-5} and Claim~\ref{clm-2.1}, we obtain:

\begin{equation}\label{eq-2-6}
\rho^2 x_{u^*}\!\leq\!\Big(m\!-\!(c\!-\!1)r|U_1|\Big)x_{u^*}\!+\!\!\sum_{v \in V^*}\!\!\!\big(2\!-\!\frac1{\beta_v}\big)rx_v\!-\!\!\!\sum_{w \in W^*}\!\!\!f(w).
\end{equation}

\begin{claim}\label{clm-2.2}
Let $x_u=\lambda_u x_{u^*}$ for $u\in U_2$. 
Then, we have $d_{W^*}(u)\geq(\lambda_u\!-\!\frac{c}{4.5})\rho\!-\!(r\!+\!\lambda_u)$.
\end{claim}

\begin{proof}
Given $u\in U_2$,
By the definitions of $U_2$ and $V^*$, 
we know that $d_{W\setminus W^*}(u)\leq\frac{c \rho}{4.5}$, 
and there exists at least one vertex $v_0\in N_{V^*}(u)$ such that $x_{v_0}\leq x_u$. 
As noted earlier, $d_U(u)\leq r$. 
From the eigenvalue equation, we have
\begin{equation}\label{eq-2-7}
\rho x_u=x_{v_0}\!+\!\!\!\sum_{v\in N(u)\setminus\{v_0\}}\!\!\!\!x_v\leq x_{u}\!+\!x_{u^*}\!+\big(r\!-\!1\!+\!d_W(u)\big)x_{u^*}.
\end{equation}
Substituting $x_u=\lambda_ux_{u^*}$ into (\ref{eq-2-7}) and dividing both sides by $x_{u^*}$, we deduce
$d_W(u)\geq\lambda_u\rho\!-(\!r\!+\!\lambda_u)$.
Since $d_{W\setminus W^*}(u)\leq\frac{c \rho}{4.5}$,
it follows that 
$d_{W^*}(u)\geq(\lambda_u\!-\!\frac{c}{4.5})\rho\!-\!(r\!+\!\lambda_u)$.
\end{proof}

Recall that $V^* = \{ v \in U : uv \in E^*_u \text{ for some } u \in U_2 \}$. 

\begin{claim}\label{clm-2.3}
Every $v\in V^*$ has at least $(\beta_v\!-\!\frac{c}{4.5})\rho\!-\!2r\!+\!(1\!-\!\beta_v)$ non-neighbors in $W^*$.
\end{claim}

\begin{proof}
Recall that $\beta_vx_{u^*}=\max_{u\in N_{U_2}(v)}x_u\geq\frac12x_{u^*}$. Let 
$u\in\!N_{U_2}(v)$ such that $x_{u}\!=\!\beta_vx_{u^*}.$
Since $G$ is $B_{r+1}$-free, we can see that $|N_W(u)\cap N_W(v)|\leq r-1$. 
Consequently, the number of non-neighbors of $v$ in $W^*$ is at least
$d_{W^*}(u)\!-\!(r\!-\!1).$
Moreover, by Claim~\ref{clm-2.2}, we have $d_{W^*}(u)\!\geq\!(\beta_v\!-\!\frac{c}{4.5})\rho\!-\!(r\!+\!\beta_v)$.
Thus, the claim follows.
\end{proof}

Based on Claim \ref{clm-2.3}, we know that if $V^*\neq\emptyset$, then $W^*$ is also non-empty.

\begin{claim}\label{clm-2.4}
If $V^*\neq\emptyset$ and $m\geq (9r)^2$, then for any $w \in W^*$, we have
$
f(w)\geq\frac{1}{4} \sum_{v \in \overline{N}_{U}(w)} x_v,
$
where $\overline{N}_{U}(w)$ is the set of non-neighbors of $w$ in $U$.
\end{claim}

\begin{proof}
Recall that $f(w)$ is defined as $d_U(w)(x_{u^*}\!-\!x_w)x_{u^*}\!+\!\frac{1}{2}d_W(w)x_{u^*}.$
We now consider three cases based on the vertices in $W^*$.

\textbf{Case 1:} $d_W(w) x_{u^*} \geq \sum_{v \in \overline{N}_{U}(w)} x_v$. 

Now, it is easy to see that
$f(w) \geq \frac{1}{2} d_W(w) x_{u^*} \geq \frac{1}{2} \sum_{v \in \overline{N}_{U}(w)} x_v.$

\textbf{Case 2:} $d_W(w) x_{u^*} < \sum_{v \in \overline{N}_{U}(w)} x_v$ and $d_U(w) \geq \frac{\rho}2$. 

Notice that $V^*\subseteq U$. From the eigenvalue equation, we know that
\begin{equation}\label{eq-2-8}
\rho (x_{u^*}\!-\!x_w) =\!\!\!\!\sum_{v \in \overline{N}_{U}(w)}\!\!\!\!x_v-\!\!\!\!\sum_{v \in N_W(w)}\!\!\!\!x_v 
\geq \!\!\!\!\sum_{v \in\overline{N}_{U}(w)}\!\!\!\!x_v-d_W(w) x_{u^*}.
\end{equation}
Since $d_U(w) \geq \frac{\rho}2$, it follows that
\[
d_U(w)(x_{u^*}\!-\!x_w) \geq \frac{1}{2}\sum_{v \in \overline{N}_{U}(w)}\!\!\!\!x_v - \frac{1}{2}d_W(w) x_{u^*},
\]
and hence,
$f(w)\geq\frac{1}{2} \sum_{v \in \overline{N}_{U}(w)}x_v.$

\textbf{Case 3:} $d_W(w) x_{u^*} < \sum_{v \in \overline{N}_{U}(w)} x_v$ and $d_U(w)<\frac{\rho}2$. 

Let $\sum_{v \in N_U(w)} x_v=\alpha x_{u^*}$. 
By the definition of $W^*$, we know that $x_w \geq (1 -\frac {4.5r}{\rho}) x_{u^*}$ for each $w \in W^*$. 
Therefore, the eigenvalue equation gives:
$(\rho\!-\!4.5r)x_{u^*}\leq\rho x_w \leq \alpha x_{u^*}\!+\!d_W(w) x_{u^*},$
which implies that $d_W(w) \geq \rho\!- \!\alpha \!-\! 4.5r $.

On the other hand, 
since $\alpha x_{u^*}=\sum_{v \in N_U(w)} x_v\leq d_U(w) x_{u^*}$, we have $\alpha\leq d_U(w)<\frac{\rho}2$. 
In view of inequality (\ref{eq-2-8}), we obtain:
\[
d_U(w)(x_{u^*}\!-\!x_w) \geq \frac{\alpha}{\rho}\!\!\!\sum_{v \in \overline{N}_{U}(w)}\!\!\!\!x_v-\frac{\alpha}{\rho}d_W(w) x_{u^*}.
\]
It follows that 
\[
f(w)\!\geq\!\frac{\alpha}{\rho}\!\!\!\sum_{v \in \overline{N}_{U}(w)}\!\!\!\!x_v
\!+\!\big(\frac{1}{2}\!-\!\frac{\alpha}{\rho}\big)d_W(w) x_{u^*}\!\geq\!\frac{\alpha}{\rho}\!\!\!\sum_{v \in \overline{N}_{U}(w)}\!\!\!\!x_v
\!+\!\big(\frac{1}{2}\!-\!\frac{\alpha}{\rho}\big)\big(\rho\!-\!\alpha\!-\!4.5r\big)x_{u^*}.
\]

Now, to conclude that $f(w)\geq\frac{1}{4}\sum_{v\in\overline{N}_{U}(w)}x_v$, it suffices to show that
\begin{equation}\label{eq-2-9}
\big(\frac{1}{2}\!-\!\frac{\alpha}{\rho}\big)\big(\rho\!-\!\alpha\!-\!4.5r\big)x_{u^*}\geq\big(\frac14\!-\!\frac{\alpha}{\rho}\big)\!\!\!\!\sum_{v\in\overline{N}_{U}(w)}\!\!\!\!x_v.
\end{equation}
Since $\rho\geq9r$ and $\alpha<\frac \rho2$, we have $(\frac{1}{2}\!-\!\frac{\alpha}{\rho})(\rho\!-\!\alpha\!-\!4.5r)\!>\!0$.
If $\alpha\geq\frac{\rho}4$, then inequality (\ref{eq-2-9}) holds trivially.
Next, assume that $\alpha<\frac{\rho}4$. 

Notice that $\sum_{v \in U} x_v=\rho x_{u^*}$ and $\sum_{v \in N_U(w)} x_v=\alpha x_{u^*}$. 
Thus, $\sum_{v\in\overline{N}_U(w)}x_v\!=\!(\rho\!-\!\alpha) x_{u^*}$.
Therefore, we only need to demonstrate:
\[
\big(\frac{1}{2}\!-\!\frac{\alpha}{\rho}\big)\big(\rho\!-\!\alpha\!-\!4.5r\big)\geq\big(\frac14\!-\!\frac{\alpha}{\rho}\big)\big(\rho\!-\!\alpha\big),
\]
or equivalently: 
\begin{equation}\label{eq-2-10}
\rho\!\geq\!9r\!+\!\big(1\!-\!\frac{18r}{\rho}\big)\alpha.
\end{equation}
Note that $\rho\geq9r$ and $\alpha<\frac{\rho}4$. 
Whether or not $\rho\leq18r$,  
it is straightforward to see that (\ref{eq-2-10}) holds, thus completing the proof.
\end{proof}

Furthermore, by Claim~\ref{clm-2.4}, we have
\begin{equation}\label{eq-2-11}
\sum_{w \in W^*}\!\!f(w)\geq\frac{1}{4}\sum_{w \in W^*}\sum_{v \in \overline{N}_{V^*}(w)}\!\!x_v
=\frac{1}{4} \sum_{v \in V^*}\sum_{w \in \overline{N}_{W^*}(v)}\!\!x_v.
\end{equation}

\begin{proof}[\bf{Proof of Theorem \ref{Main-thm-1}}]
Let \( G \) be a \( B_{r+1} \)-free graph on \( m \geq (9r)^2\) edges satisfying
\( \rho(G) \geq \sqrt{m} \).
Additionally, let $c=1$ in the definitions of $U_1$ and $U_2$.
It suffices to prove that \( G \) must be a complete bipartite graph.

First, we aim to show that \( V^* = \emptyset \).  
Suppose, for contradiction, that \( V^* \neq \emptyset \).
Recall that $\frac12\leq \beta_v\leq1$ for $v\in V^*$. 
By Claim~\ref{clm-2.3}, 
every $v\in V^*$ has at least $(\beta_v\!-\!\frac{1}{4.5})\rho\!-\!2r\!+\!(1\!-\!\beta_v)$ non-neighbors in $W^*$.
Note that $\rho \geq \sqrt{m} \geq 9r$.
Thus, from (\ref{eq-2-11}) we have
\begin{equation}\label{eq-2-12}
\sum_{w \in W^*}\!\!f(w)\geq\frac{1}{4}\sum_{v \in V^*}\!\!\big((9\beta_v\!-\!4)r+(1\!-\!\beta_v)\big)x_v>\!\!\sum_{v\in V^*}\!\!\big(2\!-\!\frac1{\beta_v}\big)r x_v,
\end{equation}
where the second inequality follows from the fact that $9\beta_v+\frac4{\beta_v}>12$ (unless $\beta_v=\frac23$).

Note that $\rho^2 \geq m $ and $c=1$.  
Rearranging terms of (\ref{eq-2-6}), we have

\begin{equation*}
\sum\limits_{v \in V^*}\!\!\!\big(2\!-\!\frac1{\beta_v}\big)rx_v
\!\geq\! \sum\limits_{w \in W^*}f(w),
\end{equation*}
which leads to a contradiction with the inequality in (\ref{eq-2-12}). 
Therefore, $V^*=\emptyset$.

By the definition of $V^*$,
this immediately implies that $U_2 = \emptyset$. 
Next, we establish that $U_1 = \emptyset$.
Assume that $U_1 \neq \emptyset$. Then, by Claim \ref{clm-2.1}, we obtain

\[
\sum_{w \in W \setminus W^*}\!\!f(w)>  r |U_1|x_{u^*}.
\]
Combining this inequality with \eqref{eq-2-2}--\eqref{2-eq-5}, 
we deduce that \eqref{eq-2-6} is also strict. 
This yields:
\[
\sum_{v \in V^*}\!\!\big(2\!-\!\frac1{\beta_v}\big)rx_v>\!\!\sum_{w \in W^*}\!\!f(w).
\]
Since $V^*=\emptyset$, the left-hand side of the above inequality equals zero, 
while the right-hand side is clearly non-negative.
This leads to a contradiction. Hence, $U_1 = \emptyset$.

Since $U_1 = \emptyset$ and $U_2 = \emptyset$, we get that $x_u + x_v < x_{u^*}$ for any $uv \in E(U)$.
We now proceed to show that $e(U) = 0$.  
Assume, for contradiction, that $e(U) \neq 0$. Then:
\[
\sum_{uv \in E(U)} \!\!(x_u\! +\! x_v)\! <\!\!\sum_{uv \in E(U)}\!\! x_{u^*} \!=\! e(U) x_{u^*}.
\]
Substituting this inequality and $\sum_{w \in W} d_U(w) x_w\leq  e(U,W) x_{u^*}$ into equation \eqref{eq-2-1}, we obtain:
\begin{equation*}
\rho^2 x_{u^*}\!<\!d(u^*)x_{u^*}\!+\!e(U) x_{u^*}\!+\!e(U,W) x_{u^*}\!=\!\big(m\!-\!e(W)\big) x_{u^*}\!\leq\!m x_{u^*},
\end{equation*}
which contradicts the fact that $\rho^2 x_{u^*} \geq m x_{u^*}$. Therefore, $e(U) = 0$.

Finally, substituting $e(U) = 0$ and $\sum_{w \in W} d_U(w) x_w\leq  e(U,W) x_{u^*}$ into \eqref{eq-2-1} gives:
\begin{equation}\label{eq-2-8}
\rho^2 x_{u^*}\!\leq\!d(u^*)x_{u^*}\!+\!e(U,W) x_{u^*}\!=\!(m\!-\!e(W)) x_{u^*}\!\leq\!m x_{u^*}.
\end{equation}
Since $\rho^2 x_{u^*} \geq m x_{u^*}$, we have $\sum_{w \in W} d_U(w) x_w= e(U,W) x_{u^*}$, and all inequalities in \eqref{eq-2-8} must hold as equalities. 
Therefore, $e(W) = 0$, and \(x_w = x_{u^*}\) for all \(w \in W\). 
This implies that $N(w)=N(u^*)$ for all $w \in W$.
Hence, $G$ is a complete bipartite graph.
\end{proof}

\section{Proof of Theorem \ref{Main-thm-2}}\label{sec-3}

Let \( G\) be an arbitrary $B_{r+1}$-free graph with \( m \geq (240r)^2\) edges, such that
its spectral radius $\rho$ satisfies: 
$\rho^2\geq m-1+\frac{2}{\rho-1}$.
To demonstrate Theorem \ref{Main-thm-2},
it suffices to show that $G\cong S^{+}_{m,s}$ for some $s\in\{1,2,\ldots,r\}$.

Let \( \rho = \rho(G) \) and let \( \mathbf{x} \) be the Perron vector of \( G\), 
with coordinates \( x_u \) for \( u \in V(G) \). Let \( u^* \) be a vertex with maximum Perron weight, i.e., 
\( x_{u^*} = \max\{x_u : u \in V(G)\} \). Define \( U = N(u^*) \) and \( W = V(G) \setminus (\{u^*\} \cup U) \). 
The definitions of \( U_1, U_2, E_1, E_2, E_3 \) are the same as in Section 2, 
with \( c = 2 \) adopted in the definitions of \( U_1 \) and \( U_2 \).
Therefore,  
inequalities (\ref{eq-2-1})-(\ref{eq-2-6}) and Claims \ref{clm-2.1}-\ref{clm-2.4} hold true in this section as well.

Notice that $c=2$ and $\rho^2> m-1$. Combining these with inequality (\ref{eq-2-6}), we obtain  
\begin{equation*}
(m-1)x_{u^*}\!<\!\Big(m\!-\!r|U_1|\Big)x_{u^*}\!+\!\!\sum_{v \in V^*}\!\!\!\big(2\!-\!\frac1{\beta_v}\big)rx_v\!-\!\!\!\sum_{w \in W^*}\!\!\!f(w),
\end{equation*}
which leads to the following inequality:
\begin{equation}\label{eq-2-14}
\sum_{w \in W^*}\!\!\!f(w)+(r|U_1|-1)x_{u^*}\!\!<\!\!\sum_{v \in V^*}\!\!\!\big(2\!-\!\frac1{\beta_v}\big)rx_v.
\end{equation}
Since \( m \geq (240r)^2\), we have $\rho>\sqrt{m-1}>239r$, and thus
$$\big(\beta_v\!-\!\frac{2}{4.5}\big)\rho\!-\!2r\geq\big(239\beta_v\!-\!109\big)r\geq 100\big(2-\frac{1}{\beta_v}\big)r.$$
Thus, by Claim \ref{clm-2.3}, every $v\in V^*$ has at least $100(2-\frac{1}{\beta_v})r$ non-neighbors in $W^*$. 
Therefore, by inequality (\ref{eq-2-11}), we have
\begin{equation}\label{eq-2-15}
\sum_{w \in W^*}\!\!f(w)\geq
\frac{1}{4} \sum_{v \in V^*}\sum_{w \in \overline{N}_{W^*}(v)}\!\!x_v\geq 25 \sum_{v \in V^*}\big(2-\frac{1}{\beta_v}\big)rx_v.
\end{equation}

\begin{claim}\label{clm-3.1}
We have $U_1=\emptyset$, and thus $E_1=\emptyset$.
\end{claim}

\begin{proof}
We prove this claim by contradiction. Suppose $U_1 \neq \emptyset$. Then, $r|U_1| \geq 1$. 
Substituting this into (\ref{eq-2-14}), we obtain:
$\sum_{w\in W^*}f(w)\!<\!\sum_{v\in V^*}(2\!-\!\frac{1}{\beta_v})rx_v,$
which contradicts (\ref{eq-2-15}). 
Therefore, we have $U_1 = \emptyset$, and consequently, $E_1=\emptyset.$
\end{proof}

By Claim \ref{clm-3.1} and inequality (\ref{eq-2-14}), we have
\begin{equation}\label{eq-2-16}
\sum_{w \in W^*}\!\!\!f(w)\!\!<\!\!\sum_{v \in V^*}\!\!\!\big(2\!-\!\frac1{\beta_v}\big)rx_v\!+\!x_{u^*}.
\end{equation}
Based on inequality (\ref{eq-2-16}), we will proceed to establish a claim regarding the edges in $E_2$.

\begin{claim}\label{clm-3.2}
We have $\sum_{uv \in E_2}(x_u\!+\!x_v)<(|E_2|+\frac{1}{24})x_{u^*}.$
\end{claim}

\begin{proof}
From inequality (\ref{2-eq-4}), we know that
\[\sum_{uv \in E_2}\!\! (x_u\!+\!x_v)\leq |E_2|x_{u^*}\!+\sum_{v \in V^*}\!\!\!\big(2\!-\!\frac1{\beta_v}\big)rx_v.\]
It suffices to show
$\sum_{v \in V^*}\!\!\big(2\!-\!\frac1{\beta_v}\big)rx_v<\frac{1}{24}x_{u^*}.$

Assume, for the sake of contradiction, that 
$\sum_{v \in V^*}\!\!\big(2\!-\!\frac1{\beta_v}\big)rx_v\geq\frac{1}{24}x_{u^*}.$ 
Then, by inequality (\ref{eq-2-15}), we obtain:
\begin{equation*}
\sum_{w \in W^*}\!\!f(w)\geq
25\!\!\sum_{v \in V^*}\!\!\!\big(2-\frac{1}{\beta_v}\big)rx_v\geq x_{u^*}\!+\!\sum_{v \in V^*}\!\!\!\big(2-\frac{1}{\beta_v}\big)rx_v,
\end{equation*}
which contradicts inequality (\ref{eq-2-16}).
Therefore, the claim is proven.
\end{proof}

Recall that $E_1=\emptyset$, and $x_u+x_v<x_{u^*}$ for each $uv\in E_3$.
Thus,
by equality (\ref{eq-2-2}) and Claim \ref{clm-3.2}, 
we obtain $\rho^2 x_{u^*}\!<\!\big(m\!+\!\frac{1}{24}\big)x_{u^*}\!-\!\sum_{w \in W}\!\!f(w)$.
Since $\rho^2\!>\!m\!-\!1$, we conclude that:
\begin{equation}\label{eq-2-17}
\sum_{w \in W}\!\!f(w)\!<\!\big(1\!+\!\frac{1}{24}\big)x_{u^*}.
\end{equation}

Notice that $f(w)\!=\!d_U(w)(x_{u^*}\!-\!x_w)\!+\!\frac12d_W(w)x_{u^*}.$
From inequality (\ref{eq-2-17}), we can observe that $e(W)=\frac12\sum_{w \in W}d_W(w)\leq1$.

Recall that $W^* = \{ w \in W :~x_w \geq (1 - \frac{4.5r}{\rho}) x_{u^*} \}$.
For any $w\in W\setminus W^*$, we have 
$f(w)\geq \frac{4.5r}\rho d_U(w),$ and thus, (\ref{eq-2-17}) implies that $d_U(w)\leq \frac\rho4.$
It follows that $\rho x_w\leq (\frac\rho4+1)x_{u^*}\leq\frac{\rho}3x_{u^*}$.
Therefore, we further have $x_w\leq \frac13x_{u^*}$ and $f(w)\geq \frac23d_U(w)+\frac12d_W(w)$ for each $w\in W\setminus W^*$.
Hence, from inequality (\ref{eq-2-17}) and $e(W)\leq1$, we can also conclude that 
$|W\setminus W^*|\leq 1$, and if there exists $w_0\in W\setminus W^*$, then 
\begin{equation}\label{eq-18}
d(w_0)=1.
\end{equation}

\begin{claim}\label{clm-3.3}
If $E(U)\neq \emptyset$, then \( x_u < \frac{1}{3}x_{u^*} \) for each \( u \in U \).
\end{claim}

\begin{proof}
Assume, for the sake of contradiction, that there exists \( u_0 \in U \) such that \( x_{u_0} = \max_{u \in U} x_u \geq \frac{1}{3}x_{u^*} \). 
Let \( \lambda = x_{u_0}/{x_{u^*}} \), where \( \frac{1}{3} \leq \lambda \leq 1 \).
Let \(u_1v_1 \in E(U) \) with \( x_{u_1} \geq x_{v_1} \). 

On the other hand, recall that $e(W)\leq1$ and $|W\setminus W^*|\leq 1$.
Let $w_0$ be the unique possible vertex in $W\setminus W^*$,
and let $w_1w_2$ be the unique possible edge in $E(W)$.
We now define $W_0=W\setminus \{w_0,w_1,w_2\}$. 
Then, $W_0\subseteq W^*$, and $d_W(w)=0$ for any $w\in W_0$.

Note that $d_U(u_0)\leq r$. By the eigenvalue equation for \( u_0 \), we have
\[
\lambda\rho x_{u^*}=\rho x_{u_0} \leq \big(r + 4 + d_{W_0}(u_0) \big) x_{u^*},
\]
which implies \( d_{W_0}(u_0) \geq \lambda\rho\!-\!r -\!4\geq \lambda\rho\!-\!5r\).
Hence, we have \( |W_0| \geq \lambda \rho - 5r\).

Define \( W_1 = N_{W_0}(u_1) \cap N_{W_0}(v_1) \), \( W_2 = N_{W_0}(u_1) \setminus W_1\), 
and \( W_3 = W_0\setminus (W_1 \cup W_2) \). 
Since \( G^* \) is \( B_{r+1} \)-free, 
we have \( |W_1| \leq r\), which implies that 
\( |W_0\setminus W_1|\geq\lambda \rho \!-\! 6r\).
For each \( w \in W_0\), we know that $w\in W^*$. Thus,
\( x_w \geq (1-\frac{4.5r}{\rho}) x_{u^*} \), 
and the eigenvalue equation gives \( \rho x_w \leq d(w) x_{u^*} \). 
Therefore, \( d(w)=d_U(w)\geq \rho - 4.5r \) for $w\in W_0$.

Notice that \( f(w) = d(w)(x_{u^*} - x_w) \) for each \( w \in W_2\cup W_3\).
If \( w \in W_2\), we have \( f(w) \geq (\rho - 4.5r)\frac{x_{v_1}}{\rho}\geq(1-\frac{6r}\rho)x_{v_1}. \) 
We also have \( f(w) \geq (1-\frac{6r}\rho)x_{u_1}\) for \( w \in W_3\).
Summing these two inequalities and noting that \( x_{u_1} \geq x_{v_1} \), we obtain:
\[
\sum_{w \in W_2\cup W_3}\!\!\!f(w) \geq |W_0\setminus W_1|\big(1\!-\!\frac{6r}{\rho}\big)x_{v_1}
\geq (\lambda \rho\!-\!6r)\big(1\!-\!\frac{6r}{\rho}\big) x_{v_1}.
\]
Since \( \rho>239r \), we have \( 2r\big( 1 \!-\! \frac{6r}{\rho} \big) \geq 1 \).
It follows that
\[
(\lambda \rho \!-\! 6r) \big( 1\! -\! \frac{6r}{\rho} \big) x_{v_1} 
\geq (\lambda \rho\!-\! 8r\!) \big( 1\! -\! \frac{6r}{\rho} \big) x_{v_1} \!+\! x_{v_1}.
\]
Observe that \( \rho x_{v_1} \geq x_{u_1} + x_{u^*} \). Combining this with the above two inequalities gives:
\begin{equation}\label{eq-19}
\sum_{w \in W_2\cup W_3}\!\!\!f(w)\geq \big( \lambda\! -\! \frac{8r}{\rho} \big) \big( 1\! -\! \frac{6r}{\rho} \big) (x_{u_1}\! +\! x_{u^*}) \!+\! x_{v_1}.
\end{equation}

We now verify the following inequality:
\begin{equation}\label{eq-20}
\big(\lambda \!-\!\frac{8r}{\rho} \big) \big(1\!-\! 
\frac{6r}{\rho}\big)(x_{u_1}\!+\!x_{u^*})\geq x_{u_1}\!+\!\frac{1}{24}x_{u^*}.
\end{equation}
Given that \( x_{u_1}\leq x_{u_0} = \lambda x_{u^*} \), 
where $\lambda\leq1$, it suffices to check the simplified inequality:
$(\lambda\!-\!\frac{8r}{\rho})(1\!-\!\frac{6r}{\rho})(\lambda\!+\!1)\geq\lambda\!+\!\frac{1}{24}.$
This clearly holds under the conditions \(\rho>239r \) and \(\lambda \geq\frac13\).
Hence, inequality \eqref{eq-20} is established.

Now, combining \eqref{eq-19} and \eqref{eq-20}, we obtain 
$\sum_{w \in W} f(w)\!\geq\!x_{u_1}\!+\!x_{v_1}\!+\!\frac{1}{24}x_{u^*}.$
On the other hand, by Claim \ref{clm-3.2}, we know that
$$\sum_{uv\in E(U)}\!\!\!(x_u\!+\!x_v)\leq\big(e(U)\!-\!1\!+\!\frac1{24}\big)x_{u^*}\!+\!(x_{u_1}\!+\!x_{v_1}).$$
From equality \eqref{eq-2-2}, it follows that
\[
\rho^2x_{u^*}\leq\big(m\!-\!1\!+\!\frac1{24}\big)x_{u^*} 
\!+\!(x_{u_1}\!+\!x_{v_1})\!-\!\sum_{w \in W}f(w)\leq(m\!-\!1)x_{u^*},
\]
which contradicts the assumption that \( \rho^2>m-1 \). 
Therefore, \( x_u < \frac{1}{3}x_{u^*} \) for all \(u\in U\).
\end{proof}

\begin{claim}\label{clm-3.4}
$W = W^*$ and $E(W) = \emptyset$.
\end{claim}

\begin{proof}
As discussed earlier, 
we know that $e(W)\leq1$ and $|W\setminus W^*|\leq1$.
First, 
we prove that $W = W^*$. Assume, for the sake of contradiction, that there exists $w_0\in W \setminus W^*$. 
Then, by \eqref{eq-18}, we have $d(w_0)=1$, which implies 
$x_{w_0}\leq\frac1{\rho}x_{u^*}$.

Let $w$ be the unique neighbor of $w_0$.
If $w\in W$, then $f(w)+f(w_0)\geq x_{u^*}$, and the unique edge $ww_0\in E(W)$ is a pendant edge.
If $w\in U$, then $f(w_0)\geq(1-\frac1\rho)x_{u^*}$.
Furthermore, we asserts that $E(W)=\emptyset$. To see this, suppose that there exists $w_1w_2\in E(W)$.
Then, we have $\sum_{i=0}^2f(w_i)\geq(2-\frac1{\rho})x_{u^*}$, which contradicts inequality (\ref{eq-2-17}).
Since $G$ is non-bipartite, we conclude that in both cases, $E(U)\neq\emptyset.$
Now, by Claim \ref{clm-3.3}, $x_{u}+x_{v}<\frac23x_{u^*}$ for any $uv\in E(U)$.
From \eqref{eq-2-2}, we can deduce the following inequality:
\[
\rho^2 x_{u^*} \leq \big(m\!-\!\frac13e(U)\big)x_{u^*}\!-\!\sum_{w \in W}f(w)\leq\big(m\!-\!\frac43\!+\!\frac1{\rho}\big)x_{u^*},
\]
which contradicts the assumption that 
\( \rho^2 > m - 1 \). 
Therefore, we have $W = W^*$.

We next show that $E(W) = \emptyset$. 
Suppose, for contradiction, that there exists an edge $w_1w_2 \in E(W)$. 
Then $w_1, w_2 \in W^*$, and so
$x_{w_1}\!+\!x_{w_2}\!\geq\!2(1\!-\!\frac{4.5r}{\rho})x_{u^*}\!>\!1.5x_{u^*}$.
On the other hand, $w_1$ and $w_2$ have at most $r$ common neighbors in $U$. 
Note that $\sum_{u\in U}x_u=\rho x_{u^*}$. Thus,
\[
\rho(x_{w_1} + x_{w_2}) \leq x_{w_1} + x_{w_2} + \rho x_{u^*} + r x_{u^*},
\]
which implies $x_{w_1} + x_{w_2} \leq \frac{\rho + r}{\rho - 1} x_{u^*} \leq 1.5x_{u^*}$, a contradiction.
Hence, $E(W) = \emptyset$. 
\end{proof}

\begin{claim}\label{clm-3.5}
Let $\rho_0=\rho(S^{+}_{m,r})$.
Then
$\rho_0^2 = m - 1 + \frac{2r}{\rho_0 - 1} $.
\end{claim}

\begin{proof}
Let \( \mathbf{x} \) be its Perron vector with coordinates \( x_v \) for 
each \( v \in V(S^{+}_{m,r}) \). Let \( u^* \) be a vertex attaining the maximum Perron weight, 
that is, \( x_{u^*} = \max\{x_v: d v \in V(S^{+}_{m,r})\} \). 
Define \( U \!=\! N(u^*) \) and \( W \!=\! V(S^{+}_{m,r}) \setminus \bigl( \{u^*\} \cup U \bigr) \).
Then, $|W|=r-1$.

Observe that $S^{+}_{m,r}$ contains exactly one edge $u_1v_1$ in $E(U)$ and satisfies $N(w)=N(u^*)$ for each $w\in W$. 
Thus, $x_w = x_{u^*}$ and $f(w) = 0$ for all $w\in W$ (where $f(w) = d(w)(x_{u^*} - x_w)$ as defined earlier).
From \eqref{eq-2-2}, we derive the following equation:
\begin{equation}\label{eq-21}
\rho_0^2 x_{u^*} = (m - 1)x_{u^*} + (x_{u_1} + x_{v_1}).
\end{equation}
Now, applying the eigenvalue equation to the vertices $u_1$ and $v_1$, we obtain:
\begin{equation*}
\rho_0(x_{u_1} + x_{v_1}) = (x_{u_1} + x_{v_1}) + 2r x_{u^*},
\end{equation*}
which simplifies to $x_{u_1} + x_{v_1} = \frac{2r}{\rho_0 - 1}x_{u^*}$. 
Substituting this into \eqref{eq-21}, we can deduce that 
$\rho_0^2 = m - 1 + \frac{2r}{\rho_0 - 1}$, as claimed.
\end{proof}

In the following, we present the proof of Theorem \ref{Main-thm-2}.
Under the assumption that
$\rho^2\geq m-1+\frac{2}{\rho-1}$,
it suffices to show that $G\cong S^{+}_{m,s}$ for some $s\in\{1,2,\ldots,r\}$.

\begin{proof}[\bf{Proof of Theorem \ref{Main-thm-2}}]

Since \( G^* \) is non-bipartite and \( e(W) = 0 \), we conclude that \( e(U) \geq 1 \). 
Let \( u_1u_2 \in E(U) \) with \( x_{u_1} \geq x_{u_2} \). 
Define
 \( W_1 = N_W(u_1) \cap N_W(u_2) \),
 \( W_2 = N_W(u_1) \setminus W_1 \), and
 \( W_3 = W \setminus (W_1 \cup W_2)\).
By the eigenvalue equations for $u_1$ and $u_2$, we obtain:
\[
\rho(x_{u_1} + x_{u_2}) \leq (x_{u_1} + x_{u_2}) + (|W| + |W_1| + 2)x_{u^*},
\]
which rearranges to:
\begin{equation}\label{eq-22}
x_{u_1}\!+\!x_{u_2}\leq\frac{x_{u^*}}{\rho\!-\!1}\big(|W|\!+\!|W_1|\!+\!2\big).
\end{equation}
By Claim \ref{clm-3.4}, we know that $W=W^*$,
and by Claim \ref{clm-3.3}, $x_u<\frac13x_{u^*}$ for each $u\in U$.
Hence, for any \( w \in W \), we have
\[\rho\big(1\!-\!\frac{4.5r}{\rho}\big)x_{u^*}\leq \rho x_w\leq\!\!\sum_{u\in N_U(w)}\!\!\!x_u\leq\frac{1}{3}d_U(w)x_{u^*},\]
which implies \( d_U(w)\geq3(\rho - 4.5r)\geq\frac{8}{3}\rho \).
For \( w \in W_2\),  we have \( \rho(x_{u^*} - x_w) \geq x_{u_2} \), and thus
\(
f(w) = d_U(w)(x_{u^*}\!-\!x_w)\geq\frac{8}{3}x_{u_2}.
\)
Similarly, for \( w \in W_3\), we obtain
\(
f(w) = d_U(w)(x_{u^*}\!-\!x_w)\geq\frac{8}{3}x_{u_1}\).
Recall that $x_{u_1}\geq x_{u_2}$. Thus, we conclude that
\begin{equation}\label{eq-23}
\sum_{w \in W}\!f(w)\geq\!\!\!\sum_{w \in W_2\cup W_3}\!\!\!f(w) \geq\big(|W|\!-\!|W_1|\big)\cdot\frac{8}{3}x_{u_2}.
\end{equation}

Since \( G^* \) is \( B_{r+1} \)-free, we have \( |W_1| \leq r - 1 \).
Notice that $\rho>239r$.
If $x_{u_2}\geq \frac14x_{u^*}$, then we can easily obtain:
\begin{equation}\label{eq-24}
(\rho-1)x_{u_2}\geq \frac{3}{4}(|W_1|\!+\!1)x_{u^*}.
\end{equation}
If $x_{u_2}<\frac14x_{u^*}$, then 
by the eigenvalue equation for \( u_2 \), 
we have
\[
\rho x_{u_2}\!\geq\! x_{u^*}\!+\!\sum_{w\in W_1}x_w\geq \Big(1\!+\!|W_1|\big(1\!-\!\frac{4.5r}{\rho}\big)\Big)x_{u^*}
\geq \frac{3}{4}\big(|W_1|\!+\!1\big)x_{u^*}\!+\!x_{u_2},
\]
which also simplifies to inequality \eqref{eq-24}.

Now, combining inequalities \eqref{eq-23} and \eqref{eq-24}, we get:
\begin{equation}\label{eq-25}
\sum_{w\in W}\!f(w)\geq \frac{2x_{u^*}}{\rho\!-\!1}\big(|W|\!-\!|W_1|\big)\big(|W_1|+1\big).
\end{equation}
Furthermore, subtracting inequality \eqref{eq-22} from \eqref{eq-25}, we obtain:
\begin{equation}\label{eq-26}
(x_{u_1}\!+\!x_{u_2})\!-\!\sum_{w\in W}\!f(w) \leq \frac{x_{u^*}}{\rho\!-\!1}
\Big(\big(|W|\!+\!|W_1|\!+\!2\big)\!-\!2\big(|W|\!-\!|W_1|\big)\big(|W_1|\!+\!1\big)\Big).
\end{equation}

We now demonstrate that $|W|=|W_1|.$ 
Suppose to the contrary that $|W|\geq |W_1|+1$. Then, by inequality \eqref{eq-26}, we deduce that 
$(x_{u_1}\!+\!x_{u_2})\!-\!\sum_{w\in W}\!f(w)\!\leq\!\frac{x_{u^*}}{\rho\!-\!1}.$
By Claim \ref{clm-3.3}, $x_u+x_v<\frac23x_{u^*}<x_{u^*}$ for any $uv\in E(U)\setminus\{u_1u_2\}$.
Consequently, from (\ref{eq-2-2}), we obtain:
\[
\rho^2x_{u^*}\!\leq\!\big(m-1\big)x_{u^*}\!+\!(x_{u_1}\!+\!x_{u_2})\!-\!\sum_{w\in W}\!f(w)
\!\leq\! \big(m\!-\!1\!+\!\frac{1}{\rho\!-\!1}\big)x_{u^*}.
\]
This contradicts the assumption that $\rho^2\geq m-1+\frac{2}{\rho-1}.$
Thus, we have $|W|=|W_1|.$

Next, we prove that $E(U)=\{u_1u_2\}$, and for all $u\in U$,
we have
$x_{u}\leq\frac{|W_1|+2}{\rho}x_{u^*}$.
Since $|W_1|\leq r-1$, we have $d(u)\leq1+|W_1|+d_U(u)\leq2r$,
and thus $\rho x_u\leq d(u)x_{u^*}\leq2rx_{u^*}$ for each $u\in U$.
This implies that $x_u+x_v\leq\frac{4r}{\rho}x_{u^*}<\frac14x_{u^*}$ for any $uv\in E(U)$.
From equality (\ref{eq-2-2}), we obtain
$\rho^2x_{u^*}\!\leq\!\big(m-\frac34e(U)\big)x_{u^*}$.
If $e(U)\geq2$, then $\rho^2<m-1$, which leads to a contradiction.
Hence, $E(U)=\{u_1u_2\}$.
Furthermore, we have $d(u)\leq1+|W_1|+d_U(u)\leq|W_1|+2$,
and thus $x_u\leq \frac{d(u)}{\rho}x_{u^*}\leq\frac{|W_1|+2}{\rho}x_{u^*}$ for all $u\in U$.

Now, if $W_1=\emptyset$ or if $N(w)=U$ for each $w\in W_1$, 
then $G$ is isomorphic to $S^+_{m,|W_1|+1}$, and we are done.
Finally, let $|W_1|\geq1$. We show that $N(w)=U$ for each $w\in W_1$.

Assume, for contradiction, that there exists $u_0\in U$ with $0\leq d_{W_1}(u_0)\leq |W_1|-1$.
Then, $\rho x_{u_0}\geq x_{u^*}+d_{W_1}(u_0)(1-\frac{4.5r}{\rho})x_{u^*}\geq\frac12(1+d_{W_1}(u_0))x_{u^*}.$
Moreover, for any $w\in \overline{N}_{W_1}(u_0)$, we have 
$\rho(x_{u^*}\!-\!x_w)\geq x_{u_0}$. Hence, we obtain a lower bound for 
$\sum_{w\in W_1}f(w)$:
\begin{equation}\label{eq-27}
\sum\limits_{w\in \overline{N}_{W_1}(u_0)}\!\!\!d_U(w)(x_{u^*}\!-\!x_{w})
\geq\frac{x_{u^*}}{2\rho^2}\!\!\!\sum\limits_{w\in \overline{N}_{W_1}(u_0)}\!\!\!d_U(w)\big(1\!+\!d_{W_1}(u_0)\big).
\end{equation}

Now, we provide a lower bound for $d_U(w)$ when $w\in W_1$.
Recall that
$x_{u}\leq\frac{|W_1|+2}{\rho}x_{u^*}$ for all $u\in U$.
Thus, $\rho x_w\leq d_U(w)(|W_1|+2)\frac{x_{u^*}}{\rho}$
for every $w\in W_1$, and since $x_w\geq(1-\frac{4.5r}{\rho})x_{u^*}$,
we conclude that $$d_U(w)\geq\frac{\rho(\rho-4.5r)}{|W_1|+2}\geq\frac{\rho^2}{4|W_1|}.$$
Combining this with inequality \eqref{eq-27}, and noting that $0\leq d_{W_1}(u_0)\leq |W_1|-1$, we obtain:
$$\sum_{w\in W_1}\!f(w)
\geq\frac{x_{u^*}}{8|W_1|}\big(|W_1|-d_{W_1}(u_0)\big)\big(1\!+\!d_{W_1}(u_0)\big)\geq\frac{x_{u^*}}{8}.
$$
Observe that $|W_1|\leq r-1$. Then,
$x_{u_1}\!+\!x_{u_2}\leq \frac{2(|W_1|+2)}{\rho}x_{u^*}<\frac{x_{u^*}}{8}$.
From equality (\ref{eq-2-2}), we get:
\[
\rho^2x_{u^*}\!=\!\big(m-1\big)x_{u^*}\!+\!(x_{u_1}\!+\!x_{u_2})\!-\!\sum_{w\in W_1}\!f(w)
\!\leq\! \big(m\!-\!1\big)x_{u^*}.
\]
This contradicts the assumption that $\rho^2\geq m-1.$
Thus, we conclude that 
$N(w)=U$ for each $w\in W_1$, completing the proof of Theorem \ref{Main-thm-2}.
\end{proof}


\section{Concluding Remarks}\label{sec-4}

Nosal \cite{Nosal-1} demonstrated that every graph $G$ with spectral radius $\rho(G)>\sqrt{m}$ contains a triangle, 
implying that $G$ is non-bipartite.  
On the other hand, bipartite graphs cannot have book subgraphs.
Therefore, studies related to booksize are inherently restricted to non-bipartite graphs.

Theorem~\ref{Main-thm-0} implies that every non-bipartite graph $G$ with spectral radius \( \rho(G) > \sqrt{m} \)
always satisfies 
$$\mathrm{bk}(G) > \frac{1}{144} \sqrt{m}.$$
Example~\ref{exam} illustrates that for certain non-bipartite graphs with \(\rho(G)\geq\sqrt{m}\), 
the booksize can reach \(\mathrm{bk}(G) = \frac{1}{3}\sqrt{m}\). 
This observation naturally raises the question of whether the condition in Theorem~\ref{Main-thm-0} can be improved.

Theorem \ref{Main-thm-2} further implies an interesting result:  
even with a weaker spectral condition, we can still derive 
the supersaturation of booksize.  
Specifically, if \(G\) is a non-bipartite graph with \(m\) edges and spectral radius \(\rho\) such that  
\(
\rho^2 \!\geq\! m \!-\! 1 \!+\! \frac{2}{\rho\! -\! 1},
\)
then  
\[
\mathrm{bk}(G)>\frac{1}{240} \sqrt{m},
\]  
unless $G\cong S_{m,s}^+$ for some positive integer $s\leq \frac{1}{240} \sqrt{m}$.

Notably, Example~\ref{exam} presents the graph \(C_3^\square [k]\), 
a non-bipartite regular graph with \(m\) edges whose spectral radius 
\(\rho\!=\sqrt{m},\) 
for which \(\mathrm{bk}(C_3^\square [k]) = k = \frac{1}{3}\sqrt{m}\).  
This example demonstrates a substantial gap between our previously derived lower bounds and the current value, 
naturally motivating the following open question:

\begin{ques}
Excluding the graph family $\{S_{m,s}^+\}$,
what is the best possible constant \(C\) such that every non-bipartite graph \(G\),
with \(m\) edges and spectral radius \(\rho\) satisfying  
\[
\rho^2\geq m - 1 + \frac{2}{\rho \!- \! 1}
\]  
has booksize \(\mathrm{bk}(G) > C \sqrt{m}\)?  
\end{ques}

Theorem \ref{Main-thm-2} and Example~\ref{exam} indicate that $\frac1{240}\leq C\leq \frac13.$

\end{document}